\documentclass{amsart}
\usepackage{mathrsfs}
\usepackage{mathrsfs}
\usepackage{amsfonts}
\usepackage{cases}
\usepackage{latexsym}
\usepackage{amsmath}
\usepackage[arrow,matrix]{xy}
\usepackage{stmaryrd}
\usepackage{amsfonts}
\usepackage{amsmath,amssymb,amscd,bbm,amsthm,mathrsfs,dsfont}

\usepackage{fancyhdr}
\usepackage{amsxtra,ifthen}
\usepackage{verbatim}

\numberwithin{equation}{section}

\theoremstyle{plain}
\newtheorem{theorem}{Theorem}[section]
\newtheorem{lemma}[theorem]{Lemma}
\newtheorem{proposition}[theorem]{Proposition}

\theoremstyle{definition}

\begin{document}

\title[Nonlinear Jordan Derivations of Triangular Algebras]
{Nonlinear Jordan Derivations of Triangular Algebras}

\author{Zhankui Xiao}

\address{Xiao: School of Mathematical Sciences, Huaqiao University,
Quanzhou, Fujian, 362021, P. R. China}

\email{zhkxiao@gmail.com}

\begin{abstract}
In this paper we prove that any nonlinear Jordan derivation on triangular algebras
is an additive derivation. As a byproduct, we obtain that any nonlinear Jordan derivation on nest algebras
over infinite dimensional Hilbert spaces is inner.
\end{abstract}



\thanks{The author of this work is supported by a research
foundation of Huaqiao University (Grant No. 10BS323).}

\maketitle

\section{Introduction}
\label{xxsec1}

Let $\mathcal{R}$ be a commutative ring with identity and let $\mathcal{A}$ be an
$\mathcal{R}$-algebra. We denote by $x\circ y=xy+yx$ the Jordan product of elements $x,y\in \mathcal{A}$. An $\mathcal{R}$-linear map
$d: \mathcal{A}\rightarrow\mathcal{A}$ is called a {\em Jordan derivation} if $d(x\circ y)=d(x)\circ y+x\circ d(y)$.
If the linear assumption in the above definition is deleted, then the corresponding
map is said to be a nonlinear Jordan derivation. It should be remarked that there are several definitions
of (linear) Jordan derivations and all of them are equivalent as long as the algebra $\mathcal{A}$
is $2$-torsion free. We refer the reader to \cite{Br2} for more details and related topics.

In 1957, Herstein proved that every Jordan derivation on a $2$-torsion free prime ring is a derivation \cite{He}.
After that, it is of much interest for researchers to find which algebras could make a Jordan derivation degenerate
to a derivation. Cusack in \cite{Cu} and Bre\v{s}ar in \cite{Br1} independently found that the semiprime rings
satisfy the condition. Recently, Johnson proved that a continuous Jordan derivation from a $C^*$-algebra $\mathcal{A}$
into a Banach $\mathcal{A}$-bimodule is a derivation \cite{Jo}.

Let $A$ and $B$ be two unital algebras over $\mathcal{R}$ and $M$ a unital
$(A, B)$-bimodule which is faithful as a left $A$-module and also as a right $B$-module.
A \textit{triangular algebra} $\mathcal{T}$ is an associative algebra of the form
$$
\left[
\begin{array}
[c]{cc}%
A & M\\
0 & B\\
\end{array}
\right]=\left\{ \hspace{2pt} \left[
\begin{array}
[c]{cc}%
a & m\\
0 & b\\
\end{array}
\right] \hspace{2pt} \vline \hspace{2pt} a\in A, b\in B, m\in M
\hspace{2pt} \right\}
$$
under matrix-like addition and matrix-like multiplication.

As far as we know, Cheung \cite{Cheung1,Cheung2} first initiated the study of linear maps
on triangular algebras. Then Zhang and Yu \cite{ZhY} proved that every Jordan derivation
on a $2$-torsion free triangular algebra is a derivation. This result was extended to Jordan
higher derivations by the author and Wei \cite{XiaoWei}. Most recently,
Yu and Zhang \cite{YuZhang} generalized the main results of \cite{ChenZhang,Cheung2} to nonlinear Lie derivations and
then the author and Wei \cite{XW4} extended it to nonlinear Lie higher derivations. Motivated by the above
works, we will investigate nonlinear Jordan derivations on triangular algebras analogously to nonlinear Lie derivations.

\section{Main Results}\label{xxsec2}

In this section, we will study nonlinear Jordan derivations on triangular algebras and nest algebras.
From now on, without claim we always assume that {\em any algebra and any (bi-)module
considered is $2$-torsion free}.

Let $\mathcal{T}=\left[ \smallmatrix
A & M\\
0 & B\\
\endsmallmatrix
\right]$ be a triangular algebra.
Let $1_A$ and $1_B$ be the identity of the algebras $A$ and $B$
respectively, and let $I$ be the identity of $\mathcal{T}$. Denote by
$$
P=\left[
\begin{array}
[c]{cc}%
1_A & 0\\
0 & 0\\
\end{array}
\right], \hspace{8pt} Q=I-P=\left[
\begin{array}
[c]{cc}%
0 & 0\\
0 & 1_B\\
\end{array}
\right]
$$
and
$$
\mathcal{T}_{11}=P{\mathcal T}P, \hspace{6pt}
\mathcal{T}_{12}=P{\mathcal T}Q, \hspace{6pt}
\mathcal{T}_{22}=Q{\mathcal T}Q.
$$
Thus the triangular algebra $\mathcal{T}$ can be expressed as
$\mathcal{T}
=\mathcal{T}_{11}+\mathcal{T}_{12}+\mathcal{T}_{22}.$
Clearly, $\mathcal{T}_{11}$ and $\mathcal{T}_{22}$ are subalgebras of
$\mathcal{T}$ which are isomorphic to $A$ and $B$, respectively.
$\mathcal{T}_{12}$ is a $(\mathcal{T}_{11},
\mathcal{T}_{22})$-bimodule which is isomorphic to the $(A,
B)$-bimodule $M$.

The main result of this note is as follows.

\begin{theorem}\label{xxsec1.1}
Let $\mathcal{T}=\left[ \smallmatrix
A & M\\
0 & B\\
\endsmallmatrix
\right]$ be a triangular algebra and let $d$ be a nonlinear Jordan
derivation of $\mathcal{T}$. Then $d$ is an additive derivation.
\end{theorem}

First, let us characterize $d(\mathcal{T}_{11})$,
$d(\mathcal{T}_{12})$ and $d(\mathcal{T}_{22})$.

\begin{lemma}\label{xxsec2.1}
$d(0)=0$, $d(\mathcal{T}_{12})\subseteq \mathcal{T}_{12}$.
\end{lemma}

\begin{proof}
Clearly, $d(0)=d(0\circ 0)=d(0)\circ 0+0\circ d(0)=0.$

Note that $P\circ M_{12}=M_{12}$ for all $M_{12}\in\mathcal{T}_{12}$. Then
$$
d(M_{12})=d(P)\circ M_{12}+P\circ d(M_{12}).\eqno(1)
$$
Suppose that $d(P)=S_{11}+S_{12}+S_{22}$ and
$d(M_{12})=T_{11}+T_{12}+T_{22}$, where
$S_{ij},T_{ij}\in\mathcal{T}_{ij}$ for $1\leq i \leq j \leq 2$. We
have from (1) that
$$
S_{11}M_{12}+M_{12}S_{22}=0\eqno(2)
$$
and
$T_{11}=0,T_{22}=0$. Thus the lemma follows.
\end{proof}

\begin{lemma}\label{xxsec2.2}
$d(\mathcal{T}_{11})\subseteq
\mathcal{T}_{11}+\mathcal{T}_{12}$ and
$d(\mathcal{T}_{22})\subseteq
\mathcal{T}_{22}+\mathcal{T}_{12}$.
\end{lemma}

\begin{proof}
Firstly, let us prove $d(\mathcal{T}_{11})\subseteq
\mathcal{T}_{11}+\mathcal{T}_{12}$.
It is clear that $A_{11}\circ Q=0$ for all $A_{11}\in \mathcal{T}_{11}$. Thus
$$
0=d(A_{11}\circ Q)
=d(A_{11})\circ Q+A_{11}\circ d(Q).
$$
Set $d(A_{11})=S_{11}+S_{12}+S_{22}$ and
$d(Q)=T_{11}+T_{12}+T_{22}$, where
$S_{ij},T_{ij}\in\mathcal{T}_{ij}$ for $1\leq i \leq j \leq 2$. Then
the above relation yields
$$
S_{12}+S_{22}+S_{22}+A_{11}T_{11}+A_{11}T_{12}+T_{11}A_{11}=0.
$$
This leads to $2S_{22}=0$ and then $S_{22}=0$.
Therefore, $d(\mathcal{T}_{11})\subseteq
\mathcal{T}_{11}+\mathcal{T}_{12}$.

It is proved similarly that $d(\mathcal{T}_{22})\subseteq
\mathcal{T}_{22}+\mathcal{T}_{12}$.
\end{proof}

We have from $P\circ Q=0$ that  $0=d(P)\circ Q+P\circ d(Q)$. On the other hand, Lemma
\ref{xxsec2.2} and (2) imply that $d(P)\in\mathcal{T}_{12}$ since
$P\mathcal{T}Q$ is faithful as a left $P\mathcal{T}P$-module.
Therefore $d(P)\circ Q=d(P)$. Similarly we get $P\circ d(Q)=d(Q).$ Now
the following lemma is clear.

\begin{lemma}\label{xxsec2.3}
$d(P)\in\mathcal{T}_{12}$, $d(Q)\in\mathcal{T}_{12}$ and $$d(P)+d(Q)=0. \eqno(3)$$
\end{lemma}

\begin{lemma}\label{xxsec2.4}
For any $A_{11}\in\mathcal{T}_{11}$, $M_{12}\in\mathcal{T}_{12}$ and
$B_{22}\in\mathcal{T}_{22}$, we have
$$d(A_{11}+M_{12})\subseteq
\mathcal{T}_{11}+\mathcal{T}_{12}\quad {\rm and}\quad
d(B_{22}+M_{12})\subseteq
\mathcal{T}_{22}+\mathcal{T}_{12}.$$
\end{lemma}

\begin{proof}
Since for any $A_{11}\in\mathcal{T}_{11}$ and
$M_{12}\in\mathcal{T}_{12}$, we have $M_{12}=(A_{11}+M_{12})\circ
Q$. It follows from Lemma \ref{xxsec2.1} that
$d((A_{11}+M_{12})\circ Q)\in\mathcal{T}_{12}$. Furthermore,
$$
d((A_{11}+M_{12})\circ Q)=d(A_{11}+M_{12})\circ Q+(A_{11}+M_{12})\circ d(Q).
$$
We have from $d(Q)\in\mathcal{T}_{12}$ that $(A_{11}+M_{12})\circ d(Q)\in\mathcal{T}_{12}$. Thus $d(A_{11}+M_{12})\circ Q\in\mathcal{T}_{12}$. It is straightforward to compute that $d(A_{11}+M_{12})\in\mathcal{T}_{11}+\mathcal{T}_{12}$.

It is similarly proved that $d(B_{22}+M_{12})\subseteq
\mathcal{T}_{22}+\mathcal{T}_{12}.$
\end{proof}

\begin{lemma}\label{xxsec2.5}
For any $A_{11}\in\mathcal{T}_{11}$, $M_{12}\in\mathcal{T}_{12}$ and
$B_{22}\in\mathcal{T}_{22}$, we have
\begin{enumerate}
\item[(1)] $d(A_{11}M_{12})=d(A_{11})M_{12}+A_{11}d(M_{12}),$
\item[(2)] $d(M_{12}B_{22})=d(M_{12})B_{22}+M_{12}d(B_{22}).$
\end{enumerate}
\end{lemma}

\begin{proof}
(1) Note that $A_{11}M_{12}=A_{11}\circ M_{12}$ for all
$A_{11}\in\mathcal{T}_{11}$ and $M_{12}\in\mathcal{T}_{12}$.
$$
\begin{aligned}
d(A_{11}M_{12})&=d(A_{11}\circ M_{12})\\
&=d(A_{11})\circ M_{12}+A_{11}\circ d(M_{12})\\
&=d(A_{11})M_{12}+A_{11}d(M_{12}),
\end{aligned}$$
where the last equation is due to Lemmas \ref{xxsec2.1} and
\ref{xxsec2.2}.

(2) is proved similarly.
\end{proof}

\begin{lemma}\label{xxsec2.6}
For any $A_{11}\in\mathcal{T}_{11}, B_{22}\in \mathcal{T}_{22}$ and
$M_{12}, N_{12}\in\mathcal{T}_{12}$, we have
\begin{enumerate}
\item [(1)] $d(A_{11}+M_{12})=d(A_{11})+d(M_{12})$,

\item [(2)] $d(B_{22}+N_{12})=d(B_{22})+d(N_{12})$.
\end{enumerate}
\end{lemma}

\begin{proof}
(1) It is clear that $A_{11}X_{12}=(A_{11}+M_{12})\circ X_{12}$ for all
$A_{11}\in\mathcal{T}_{11}$ and $M_{12}, X_{12}\in\mathcal{T}_{12}$. Then
$$
\begin{aligned}
d(A_{11}X_{12})
&=d((A_{11}+M_{12})\circ X_{12})\\
&=d(A_{11}+M_{12})\circ X_{12}+(A_{11}+M_{12})\circ d(X_{12})\\
&=d(A_{11}+M_{12})\circ X_{12}+A_{11}d(X_{12})
\end{aligned}\eqno(4)
$$
where the last equation depends on the fact
$d(X_{12})\in\mathcal{T}_{12}$ of Lemma \ref{xxsec2.1}. Combining
(4) with Lemma \ref{xxsec2.5} (1) leads to
$$
(d(A_{11}+M_{12})-d(A_{11}))\circ X_{12}=0.\eqno(5)
$$
By Lemmas \ref{xxsec2.2} and \ref{xxsec2.4}, let us write $d(A_{11}+M_{12})-d(A_{11})=\left[\smallmatrix a & m\\
0 & 0 \endsmallmatrix \right]$. Then (5) implies that $a=0$ and thus
$$
d(A_{11}+M_{12})-d(A_{11})\in\mathcal{T}_{12}.
$$
It follows immediately that $$d(A_{11}+M_{12})-d(A_{11})=Q\circ(d(A_{11}+M_{12})-d(A_{11})).\eqno(6)$$

On the other hand, we have from Lemma \ref{xxsec2.3} and $Q\circ
d(A_{11})+d(Q)\circ A_{11}=0$ that
$$
\begin{aligned}
&\hspace{13pt}Q\circ(d(A_{11}+M_{12})-d(A_{11}))\\
&=Q\circ d(A_{11}+M_{12})-Q\circ d(A_{11})\\
&=d(Q\circ(A_{11}+M_{12}))-d(Q)\circ(A_{11}+M_{12})+d(Q)\circ A_{11}\\
&=d(M_{12})-d(Q)\circ A_{11}+d(Q)\circ A_{11}\\
&=d(M_{12}).
\end{aligned}\eqno(7)
$$
Combining (6) with (7) yields $d(A_{11}+M_{12})=d(A_{11})+d(M_{12})$.

(2) is proved similarly.
\end{proof}

\begin{lemma}\label{xxsec2.7}
$d$ is additive on $\mathcal{T}_{11}, \mathcal{T}_{12}$ and
$\mathcal{T}_{22}$, respectively.
\end{lemma}

\begin{proof}
For any $M_{12},N_{12}\in \mathcal{T}_{12}$, it follows from
Lemmas \ref{xxsec2.1}, \ref{xxsec2.2}, \ref{xxsec2.3} and
\ref{xxsec2.6} that
$$
\begin{aligned}
d(M_{12}+N_{12})&=d((P+M_{12})\circ(Q+N_{12}))\\
&=d(P+M_{12})\circ(Q+N_{12})+(P+M_{12})\circ d(Q+N_{12})\\
&=(d(P)+d(M_{12}))\circ(Q+N_{12})+(P+M_{12})\circ (d(Q)+d(N_{12}))\\
&=d(P)+d(M_{12})+d(Q)+d(N_{12})\\
&=d(M_{12})+d(N_{12}),
\end{aligned}$$
which implies that $d$ is additive on $\mathcal{T}_{11}$.

Furthermore, by Lemma \ref{xxsec2.5} we arrive at
$$
\begin{aligned}
&\hspace{13pt}d((A_{11}+B_{11})M_{12})\\
&=d(A_{11}M_{12})+d(B_{11}M_{12})\\
&=d(A_{11})M_{12}+A_{11}d(M_{12})+d(B_{11})M_{12}+B_{11}d(M_{12})
\end{aligned}\eqno(8)
$$
for all $A_{11},B_{11}\in\mathcal{T}_{11}$ and $M_{12}\in
\mathcal{T}_{12}$.

On the other hand, by Lemma \ref{xxsec2.5} again
we obtain
$$
d((A_{11}+B_{11})M_{12})=d(A_{11}+B_{11})M_{12}+(A_{11}+B_{11})d(M_{12})\eqno(9)
$$
for all $A_{11},B_{11}\in\mathcal{T}_{11}$ and $M_{12}\in
\mathcal{T}_{12}$. Combining (8) with (9) gives
$$
d(A_{11}+B_{11})M_{12}=d(A_{11})M_{12}+d(B_{11})M_{12}.\eqno(10)
$$
Since $d(\mathcal{T}_{11})\subseteq
\mathcal{T}_{11}+\mathcal{T}_{12}$ and $P\mathcal{T}Q$ is faithful
as a left $P\mathcal{T}P$-module, the relation (10) implies that
$$
d(A_{11}+B_{11})P=d(A_{11})P+d(B_{11})P.\eqno(11)
$$

Note that $d(0)=0$ by Lemma \ref{xxsec2.1}.
$$
0=d(A_{11}\circ Q)
=d(A_{11})\circ Q+A_{11}\circ d(Q)
=d(A_{11})Q+A_{11}d(Q)
\eqno(12)
$$
for all $A_{11}\in\mathcal{T}_{11}$, where the last equation is due
to Lemma \ref{xxsec2.3}. Substituting $A_{11}$ by $B_{11}$ and
$A_{11}+B_{11}$ in (12), respectively, we have
$$
0=d(B_{11})Q+B_{11}d(Q)
\eqno(13)
$$
and
$$
0=d(A_{11}+B_{11})Q+(A_{11}+B_{11})d(Q).
\eqno(14)
$$
Combining (14) with (12) and (13) we have
$$
d(A_{11}+B_{11})Q=d(A_{11})Q+d(B_{11})Q.\eqno(15)
$$
Then (11) and (15) lead to
$$d(A_{11}+B_{11})=d(A_{11})+d(B_{11}),$$ which is the
desired result.

Similarly, we can also get the additivity of $d$ on
$\mathcal{T}_{22}$.
\end{proof}

\begin{lemma}\label{xxsec2.8}
For any $A_{ij}\in\mathcal{T}_{ij}$ with $1\leq i\leq j\leq 2$, we
have
$$
d(A_{11}+A_{12}+A_{22})=d(A_{11})+d(A_{12})+d(A_{22}).
$$
\end{lemma}

\begin{proof}
It is easy to know that
$$(A_{11}+A_{12}+A_{22})\circ P=A_{11}\circ P+A_{12}\circ P=2A_{11}+A_{12}$$ and
$$
\begin{aligned}
&\hspace{13pt}d((A_{11}+A_{12}+A_{22})\circ P)\\
&=d(A_{11}+A_{12}+A_{22})\circ P+(A_{11}+A_{12}+A_{22})\circ d(P)\\
&=d(A_{11}+A_{12}+A_{22})\circ P+A_{11}\circ d(P)+A_{12}\circ d(P)+A_{22}\circ d(P).
\end{aligned}\eqno(16)
$$
Moreover, we have from $0=d(P)\circ A_{22}+P\circ d(A_{22})$ and
Lemma \ref{xxsec2.6} that
$$
\begin{aligned}
&\hspace{13pt}d(A_{11}\circ P+A_{12}\circ P)\\
&=d(A_{11}\circ P)+d(A_{12}\circ P)\\
&=d(A_{11})\circ P+A_{11}\circ d(P)+d(A_{12})\circ P+A_{12}\circ d(P)\\
&\hspace{50pt}+d(P)\circ A_{22}+P\circ d(A_{22}).
\end{aligned}\eqno(17)
$$
Combining (16) with (17) we get
$$
(d(A_{11}+A_{12}+A_{22})-d(A_{11})-d(A_{12})-d(A_{22}))\circ P=0.\eqno(18)
$$
Set $d(A_{11}+A_{12}+A_{22})-d(A_{11})-d(A_{12})-d(A_{22})=\left[\smallmatrix a & m\\
0 & b \endsmallmatrix \right]$. Then (18) implies that $a=0$ and $m=0$.

It could be obtained that $b=0$ if we consider
$(A_{11}+A_{12}+A_{22})\circ Q$ and thus the proof is completed.
\end{proof}

\begin{lemma}\label{xxsec2.10}
For any $A_{11},B_{11}\in\mathcal{T}_{11},
A_{22},B_{22}\in\mathcal{T}_{22}$, we have
\begin{enumerate}
\item[(1)]
$d(A_{11}B_{11}) =d(A_{11})B_{11}+A_{11}d(B_{11})$,

\item[(2)] $d(A_{22}B_{22})
=d(A_{22})B_{22}+A_{22}d(B_{22})$.
\end{enumerate}
\end{lemma}

\begin{proof}
(1) It follows from Lemma \ref{xxsec2.5} that
$$
d(A_{11}B_{11}M_{12}) =d(A_{11}B_{11})M_{12}+A_{11}B_{11}d(M_{12}).
\eqno(19)
$$
On the other hand,
$$
\begin{aligned}
d(A_{11}B_{11}M_{12})
&=d(A_{11}(B_{11}M_{12}))\\
&=d(A_{11})B_{11}M_{12}+A_{11}d(B_{11}M_{12})\\
&=d(A_{11})B_{11}M_{12}+A_{11}d(B_{11})M_{12}+A_{11}B_{11}d(M_{12}).
\end{aligned}\eqno(20)
$$
Combining $(19)$ with $(20)$ leads to
$$d(A_{11}B_{11})M_{12}=d(A_{11})B_{11}M_{12}+A_{11}d(B_{11})M_{12}.\eqno(21)$$

Since $d(\mathcal{T}_{11})\subseteq
\mathcal{T}_{11}+\mathcal{T}_{12}$ and $P\mathcal{T}Q$ is faithful
as a left $P\mathcal{T}P$-module, the relation (21) implies that
$$
d(A_{11}B_{11})P=d(A_{11})B_{11}P+A_{11}d(B_{11})P.\eqno(22)
$$
Furthermore, substituting $A_{11}$ by $A_{11}B_{11}$ in (12), we
have
$$
0=d(A_{11}B_{11})Q+A_{11}B_{11}d(Q). \eqno(23)
$$
Left multiplying $A_{11}$ on both sides of  (13) gives
$$
0=A_{11}d(B_{11})Q+A_{11}B_{11}d(Q). \eqno(24)
$$
Combining (23) with (24) yields
$$d(A_{11}B_{11})Q=A_{11}d(B_{11})Q. \eqno(25)$$
Note that $d(A_{11})B_{11}Q=0$. Then (25) implies that
$$
d(A_{11}B_{11})Q=d(A_{11})B_{11}Q+A_{11}d(B_{11})Q.\eqno(26)
$$
Then (22) and (26) lead to
$$d(A_{11}B_{11})=d(A_{11})B_{11}+A_{11}d(B_{11}),$$ which is the
desired result.

(2) is proved similarly.
\end{proof}

We are now in a position to prove our main theorem.
\vspace{2mm}

{\noindent}{\bf Proof of Theorem 2.1.}
Firstly we prove the additivity of $d$.
For any $T, T'\in\mathcal{T}$, let $T=A_{11}+A_{12}+A_{22}$ and
$T'=A_{11}'+A_{12}'+A_{22}'$, where $A_{ij},
A_{ij}'\in\mathcal{T}_{ij}$. Then by Lemmas \ref{xxsec2.7} and
\ref{xxsec2.8} we have
$$
\begin{aligned}
&\hspace{13pt}d(T+T')\\
&=d(A_{11}+A_{11}')+d(A_{12}+A_{12}')+d(A_{22}+A_{22}')\\
&=d(A_{11})+d(A_{11}')+d(A_{12})+d(A_{12}')+d(A_{22})+d(A_{22}')\\
&=d(A_{11}+A_{12}+A_{22})+d(A_{11}'+A_{12}'+A_{22}')\\
&=d(T)+d(T').
\end{aligned}
$$

Then, we only need to
prove that $d(XY)=d(X)Y+Xd(Y)$ for all $X, Y\in \mathcal {T}$.
Suppose that $X=X_{11}+X_{12}+X_{22}$ and $Y=Y_{11}+Y_{12}+Y_{22}$,
where $X_{ij},Y_{ij}\in\mathcal{T}_{ij}$ with $1\leq i\leq j\leq 2$.
Then it follows from Lemmas \ref{xxsec2.5}, \ref{xxsec2.8} and
\ref{xxsec2.10} that
$$
\begin{aligned}
d(XY)&=d(X_{11}Y_{11}+X_{11}Y_{12}+X_{12}Y_{22}+X_{22}Y_{22})\\
&=d(X_{11}Y_{11})+d(X_{11}Y_{12})+d(X_{12}Y_{22})+d(X_{22}Y_{22})\\
&=d(X_{11})Y_{11}+X_{11}d(Y_{11})+d(X_{11})Y_{12}+X_{11}d(Y_{12})\\
&\hspace{10pt}+d(X_{12})Y_{22}+X_{12}d(Y_{22})+d(X_{22})Y_{22}+X_{22}d(Y_{22}).
\end{aligned}\eqno(27)
$$
On the other hand, we have from Lemmas \ref{xxsec2.1}, \ref{xxsec2.2} that
$$
\begin{aligned}
& d(X)Y+Xd(Y)\\
&=d(X_{11}+X_{12}+X_{22})(Y_{11}+Y_{12}+Y_{22})\\
&\hspace{10pt}+(X_{11}+X_{12}+X_{22})d(Y_{11}+Y_{12}+Y_{22})\\
&=(d(X_{11})+d(X_{12})+d(X_{22}))(Y_{11}+Y_{12}+Y_{22})\\
&\hspace{10pt}+(X_{11}+X_{12}+X_{22})(d(Y_{11})+d(Y_{12})+d(Y_{22}))\\
&=d(X_{11})Y_{11}+d(X_{12})Y_{11}+d(X_{22})Y_{11}+d(X_{11})Y_{12}\\
&\hspace{10pt}+d(X_{12})Y_{12}+d(X_{22})Y_{12}+d(X_{11})Y_{22}+d(X_{12})Y_{22}\\
&\hspace{10pt}+d(X_{22})Y_{22}+X_{11}d(Y_{11})+X_{12}d(Y_{11})+X_{22}d(Y_{11})\\
&\hspace{10pt}+X_{11}d(Y_{12})+X_{12}d(Y_{12})+X_{22}d(Y_{12})+X_{11}d(Y_{22})\\
&\hspace{10pt}+X_{12}d(Y_{22})+X_{22}d(Y_{22})\\
&=d(X_{11})Y_{11}+d(X_{11})Y_{12}+d(X_{12})Y_{22}+d(X_{22})Y_{22}\\
&\hspace{10pt}+X_{11}d(Y_{11})+X_{11}d(Y_{12})+X_{12}d(Y_{22})+X_{22}d(Y_{22})\\
&\hspace{10pt}+d(X_{11})Y_{22}+X_{11}d(Y_{22}).
\end{aligned}\eqno(28)
$$
We claim that $d(X_{11})Y_{22}+X_{11}d(Y_{22})=0$. In fact, we know
from Lemma \ref{xxsec2.2} that
$$
0=d(X_{11}\circ Y_{22}) =d(X_{11})\circ Y_{22}+X_{11}\circ
d(Y_{22})=d(X_{11})Y_{22}+X_{11}d(Y_{22}).
$$
Then (28) implies that
$$
\begin{aligned}
& d(X)Y+Xd(Y)\\
&=d(X_{11})Y_{11}+d(X_{11})Y_{12}+d(X_{12})Y_{22}+d(X_{22})Y_{22}\\
&\hspace{10pt}+X_{11}d(Y_{11})+X_{11}d(Y_{12})+X_{12}d(Y_{22})+X_{22}d(Y_{22}).
\end{aligned}\eqno(29)
$$
Combining (27) with (29) leads to $d(XY)=d(X)Y+Xd(Y)$ for all $X,
Y\in \mathcal {T}$, which is the desired result.\qed

\bigskip

Next we study the nonlinear Jordan derivations of nest algebras.
Let $\mathbf{H}$ be a complex Hilbert space
and $\mathcal{B}(\mathbf{H})$ be the algebra of all bounded
linear operators on $\mathbf{H}$. Let $I$ be an index set. A \textit{nest} is
a set $\mathcal{N}$ of closed subspaces of $\mathbf{H}$ satisfying the
following conditions:
\begin{enumerate}
\item[(1)] $0, \mathbf{H}\in \mathcal{N}$;
\item[(2)] If $N_1, N_2\in \mathcal{N}$, then
either $N_1\subseteq N_2$ or $N_2\subseteq N_1$;
\item[(3)] If $\{N_i\}_{i\in I}\subseteq \mathcal{N}$,
then $\bigcap_{i\in I}N_i\in \mathcal{N}$;
\item[(4)] If $\{N_i\}_{i\in I}\subseteq \mathcal{N}$, then the
norm closure of the linear span of $\bigcup_{i\in I} N_i$ also lies
in $\mathcal{N}$.
\end{enumerate}
If $\mathcal{N}=\{0, \mathbf{H}\}$, then $\mathcal{N}$ is called a trivial
nest, otherwise it is called a nontrivial nest.

The \textit{nest algebra} associated with $\mathcal{N}$ is the set
$$
\tau(\mathcal{N})=\{\hspace{3pt} T\in \mathcal{B}(\mathbf{H})\hspace{3pt}|
\hspace{3pt} T(N)\subseteq N \hspace{3pt} {\rm for} \hspace{3pt}
{\rm all} \hspace{3pt} N\in \mathcal{N}\} .
$$
If there exists $N\in\mathcal{N}$ $\backslash$ $\left\{  0,\mathbf{H}\right\}
$ which is complemented in $\mathbf{H}$ then $\tau(\mathcal{N})$ is a
triangular algebra. That means every nontrivial nest algebra is a triangular
algebra. Moreover every trivial nest algebra $\tau(\mathcal{N})=\mathcal{B}(\mathbf{H})$
is a von Neumann prime algebra. We refer the reader to the book \cite{Da} for the
theory of nest algebras.

Christensen \cite{Chris} proved that every derivation of nest algebras is inner.
This is one of the most crucial results in the theory of nest algebras. The following
proposition generalized Christensen's theorem to nonlinear Jordan derivations.

\begin{proposition}
Let $\mathbf{H}$ be an infinite dimensional complex Hilbert space. Then every nonlinear Jordan derivation of nest algebras $\tau(\mathcal{N})$ is inner.
\end{proposition}

\begin{proof}
Let $d$ be a nonlinear Jordan derivation on nest algebra $\tau(\mathcal{N})$. If the nest
$\mathcal{N}$ is nontrivial, then the nest algebra $\tau(\mathcal{N})$ is a triangular algebra.
Theorem \ref{xxsec1.1} implies that $d$ is an additive derivation. If the nest
$\mathcal{N}$ is trivial, then the nest algebra $\tau(\mathcal{N})=\mathcal{B}(\mathbf{H})$
is a von Neumann prime algebra. The main results of \cite{Lu} also deduce that $d$ is an additive derivation.
By \cite[Theorem 2.5]{Ha}, $d$ is a linear derivation as $\mathbf{H}$ is infinite dimensional. Then Christensen's theorem make
the proposition proved.
\end{proof}

\bigskip


\begin{thebibliography}{}

\bibitem{Br1} M. Bre\v{s}ar, {\em Jordan derivations on semiprime rings},
Proc. Amer. Math. Soc., \textbf{104} (1988), 1003-1006.

\bibitem{Br2} M. Bre\v{s}ar, {\em Jordan derivations revisited},
Math. Proc. Camb. Phil. Soc., \textbf{139} (2005), 411-425.

\bibitem{ChenZhang} L. Chen and J. -H. Zhang, {\em Nonlinear
Lie derivations on upper triangular matrices}, Linear Multilinear
Algebra, \textbf{56} (2008), 725-730.

\bibitem{Cheung1} W. S. Cheung, {\em Commuting maps of triangular algebras},
J. London Math. Soc., \textbf{63} (2001), 117-127.

\bibitem{Cheung2} W. S. Cheung, {\em Lie derivations of triangular algebras},
Linear Multilinear Algebra, \textbf{51} (2003), 299-310.

\bibitem{Chris} E. Christensen, {\em Derivations of nest algebras}, Math. Ann.,
\textbf{229} (1977), 155-161.

\bibitem{Cu} J. M. Cusack,{\em Jordan derivations on rings},
Proc. Amer. Math. Soc., \textbf{53} (1975), 321-324.

\bibitem{Da} K. R. Davidson, {\em Nest algebras}, Pitman
Research Notes in Mathematics Series, \textbf{191}, Longman,
London/New York, 1988.

\bibitem{Ha} D. G. Han, {\em Additive derivations of nest algebras},
Proc. Amer. Math. Soc., \textbf{119} (1993), 1165-1169.

\bibitem{He} I. N. Herstein, {\em Jordan derivations of prime rings},
Proc. Amer. Math. Soc., \textbf{8} (1957), 1104-1110.

\bibitem{Jo} B. E. Johnson, {\em Symmetric amenability and the nonexistence of Lie and
Jordan derivations}, Math. Proc. Camb. Phil. Soc., \textbf{120} (1996), 455-473.

\bibitem{Lu} F. -Y. Lu, {\em Jordan derivable maps of prime rings},
Comm. Algebra, \textbf{38} (2010), 4430-4440.

\bibitem{XiaoWei} Z. -K. Xiao and F. Wei, {\em Jordan higher
derivations on triangular algebras}, Linear Algebra Appl.,
\textbf{432} (2010), 2615-2622.

\bibitem{XW4} Z. -K. Xiao and F. Wei, {\em Nonlinear Lie
higher derivations of triangular algebras}, Submitted.

\bibitem{YuZhang} W. -Y. Yu and J. -H. Zhang, {\em Nonlinear Lie
derivations of triangular algebras}, Linear Algebra Appl.,
\textbf{432} (2010), 2953-2960.

\bibitem{ZhY} J. -H. Zhang and W. -Y. Yu, {\em Jordan derivations of triangular algebras},
Linear Algebra Appl., \textbf{419} (2006), 251-255.

\end{thebibliography}
\end{document}